\documentclass[12pt, reqno]{article}
\usepackage{verbatim}

\newcommand{\Aut}{\ensuremath{\operatorname{Aut}}}

\newcommand{\id}{\ensuremath{\text{\rm id}}}

\usepackage{amsmath, amsthm, amscd, amsfonts, amssymb, graphicx, color, soul, enumerate, xcolor,  mathrsfs, latexsym}
\usepackage[bookmarksnumbered, colorlinks, plainpages,backref]{hyperref}
\hypersetup{colorlinks=true,linkcolor=red, anchorcolor=green, citecolor=cyan, urlcolor=red, filecolor=magenta, pdftoolbar=true}
\usepackage[all]{xy}

 \definecolor{cupgreen}{rgb}{0,0.498,0.208}
  \definecolor{cupblue}{rgb}{0,0,.5}
  \definecolor{cupred}{rgb}{1,0.04,0}
  \definecolor{cuppink}{rgb}{0.925,0,0.545}
  \definecolor{cupmagenta}{rgb}{0.624,0.161,0.424}
  \definecolor{cupbrown}{rgb}{0.71,0.212,0.133}

  \definecolor{cupgreen}{rgb}{0,0,0}
  \definecolor{cupblue}{rgb}{0,0,0}
  \definecolor{cupred}{rgb}{0,0,0}
  \definecolor{cuppink}{rgb}{0,0,0}
  \definecolor{cupmagenta}{rgb}{0,0,0}
  \definecolor{cupbrown}{rgb}{0,0,0}

\definecolor{TITLE}{rgb}{0,0,0}
\definecolor{AUTHOR1}{rgb}{0.00,0.59,0.00}
\definecolor{AUTHOR2}{rgb}{0.50,0.00,1.00}
\definecolor{SECTION}{rgb}{0.50,0.00,1.00}
\definecolor{THM}{rgb}{0.8,0,0.1}
\definecolor{SEC}{rgb}{0,0,1}

\normalsize
\newtheorem{theorem}{{\color{THM} Theorem}}[section]
\newtheorem{procedure}{{\color{THM} Procedure}}[section]
\newtheorem{lemma}[theorem]{{\color{THM}Lemma}}
\newtheorem{proposition}[theorem]{{\color{THM}Proposition}}
\newtheorem{corollary}[theorem]{{\color{THM}Corollary}}

\theoremstyle{definition}

\newtheorem{example}[theorem]{{\color{THM}Example}}

\newtheorem{remark}[theorem]{{\color{THM}Remark}}
\numberwithin{equation}{section}

\title{The Cost of Edge-distinguishing of the Cartesian Product of Connected Graphs \thanks{The research was partially supported by OEAD grant no. PL 08/2017. }}

\author{Aleksandra Gorzkowska\thanks{Aleksandra Gorzkowska was partially supported by the Polish Ministry of Science and Higher Education.}\\
\small AGH University, Department of Discrete Mathematics, 30-059 Krakow, Poland\\
\small\tt agorzkow@agh.edu.pl
\and Mohammad Hadi Shekarriz\\
\small Department of Mathematics, Shiraz University, Shiraz, Iran.\\
\small\tt mshekarriz@shirazu.ac.ir\\
\date {}}

\begin{document}

\maketitle

\begin{abstract}
 A graph $G$ is said to be $d$-distinguishable if there is a vertex coloring of $G$ with a set of $d$ colors which breaks all of the automorphisms of $G$ but the identity. We call the minimum $d$ for which a graph $G$ is $d$-distinguishiable the distinguishing number of $G$, denoted by $D(G)$. When $D(G)=2$, the minimum number of vertices in one of the color classes is called the cost of distinguishing of $G$ and is shown by $\rho(G)$. In this paper, we generalize this concept to edge-coloring by introducing the cost of edge-distinguishing of a graph $G$, denoted by $\rho'(G)$. Then, we consider $\rho'(K_n )$ for $n\geq 6$ by finding a procedure that gives the minimum number of edges of $K_n$ that should be colored differently to have a $2$-distinguishing edge-coloring. Afterwards, we develop a machinery to state a sufficient condition for a coloring of the Cartesian product to break all non-trivial automorphisms. Using this sufficient condition, we determine when cost of distinguishing and edge-distinguishing of the Cartesian power of a path equals to one. We also show that this parameters are equal to one for any Cartesian product of finitely many paths of different lengths. Moreover, we do a similar work for the Cartesian powers of a cycle and also for the Cartesian products of finitely many cycles of different orders. Upper bounds for the cost of edge-distinguishing of hypercubes and the Cartesian powers of complete graphs are also presented.
\end{abstract}

\section{Introduction}

We follow standard graph theory notation. A (vertex, edge or total) coloring of a graph is called \emph{distinguishing} if no automorphism but the identity preserves it. The {\it distinguishing number} of a graph $G$, denoted $D(G)$, is the smallest number $d$ such that there exists a distinguishing vertex coloring of $G$ with $d$ colors. The graph $G$ is called {\it $d$-distinguishable} if there exists a distinguishing vertex coloring with $d$ colors. Already defined by Babai in 1977 under the names \emph{asymmetric coloring} and \emph{asymmetric coloring number} \cite{Babai}, these concepts were reintroduced for vertex coloring by Albertson and Collins in 1996 \cite{ac}. Since then other variants have been studied, in particular, for edge coloring by Kalinowski and Pil\'sniak \cite{kp} and for total coloring by Kalinowski, Pil\'sniak and Wo\'zniak \cite{kpw}. The {\it distinguishing index} of a graph $G$, denoted $D'(G)$, is the least number $d$ such that there exists a distinguishing edge coloring of $G$ with $d$ colors. The {\it total distinguishing number} of a graph $G$, denoted $D''(G)$, is the least number $d$ such that there exists a total distinguishing coloring of $G$ with $d$ colors.

Hitherto, it has been proved that many classes of graphs have the distinguishing number equal to two. Among others, for paths and cycles of sufficiently large order, we have  $D(P_n) = 2$ for $n \geq 3$, and $D(C_n) = 2$ for $n \geq 6$ \cite{ac}. Furthermore, Bogstad and Cowen in \cite{bc} showed that for $k \geq 4$, every hypercube $Q_k$ of dimension $k$, which is the Cartesian product of $k$ copies of $K_2$, is $2$-distinguishable. It has also been shown by Imrich and Klav{\v{z}}ar in \cite{ik} that the distinguishing number of Cartesian powers of a connected graph $G$ is equal to two except for $K_2^2, K_3^2, K_2^3$. Moreover, it is proved by Estaji et al. in \cite{eikpt} that for every pair of connected graphs $G$ and $H$ with $|H| \leq |G| < 2^{|H|} - |H|$, we have $D(G \Box H) \leq 2$.

For families of $2$-distinguishable graphs we might attempt to minimize the number of vertices in one of the two color classes; the minimum number of vertices in the smaller color class of a distinguishing $2$-coloring of a graph $G$ is called {\it the cost of distinguishing} of $G$, and is denoted by $\rho (G)$. The idea was first proposed by Imrich, while hypercubes were the first class of $2$-distinguishable graphs whose cost of distinguishing was studied. Boutin proved that $\lceil \log_2 n \rceil +1 \leq \rho (Q_n) \leq 2 \lceil \log_2 n \rceil -1$ for $n \geq 5$ \cite{b}.



For edge colorings (as well as for total colorings), large classes of graphs have been found in \cite{gkp, kp, p} for which the  distinguishing indices are equal to two. Analogous to vertex colorings, we define the {\it cost of edge-distinguishing} of a graph $G$, denoted $\rho'(G)$, as the least number of edges in the smaller of the two color classes in a distinguishing coloring.  

After recalling some required preliminaries to the Cartesian product of graphs in the next section, we start with some results on the cost of edge-distinguishing of complete graphs. Then, we develop a machinery for a lemma which states a sufficient condition for a coloring of a Cartesian product of connected graphs to be distinguishing. Using this lemma, we investigate cost of distinguishing and of edge-distinguishing of the Cartesian products of paths and cycles and in particular, we find some bounds for cost of edge-distinguishing of hypercubes and the Cartesian powers of complete graphs. 

\section{Preliminaries}

The {\it Cartesian product} of graphs $G$ and $H$ is a graph, denoted by $G \Box H$, whose vertex set is $V(G) \times V(H)$, and two vertices $(g, h)$, $(g', h')$ are adjacent if either $g = g'$ and $h h' \in E(H)$, or $g g' \in E(G)$ and $h = h'$. We denote $G \Box G$ by $G^2$, and we recursively define the {\it $k^{\textnormal{th}}$ Cartesian power} of $G$ as $G^k = G \Box G^{k-1}$. Graph $G$ is {\it prime} with respect to the Cartesian product if it cannot be represented as the Cartesian product of two graphs non-ismorphic with $G$. We say that graphs $G$ and $H$ are {\it relatively prime} if they do not have any non-trivial common factor \cite{hik}. 

Sabidussi and Vizing proved that every connected graph has a unique prime factorization with respect to the Cartesian product (cf.~\cite{hik}). The following well-known characterization of the automorphism group of product graphs, due to Imrich (and independently to Miller), is the core of our investigations in this paper.

\begin{theorem} \cite{hik} \label{autG}
Suppose $\psi$ is an automorphism of a connected graph $G$ with prime factor decomposition $G = G_1 \Box G_2 \Box \dots \Box G_k$. Then there is a permutation $\pi$ of the set $\{1, 2, \dots , k\}$ and there are isomorphisms $\psi_i \colon G_{\pi(i)} \mapsto G_i$, $i=1, \dots, k$, such that
$$\psi(x_1, x_2, \dots, x_k) = (\psi_1 (x_{\pi(1)}), \psi_2 (x_{\pi(2)}), \dots, \psi_r (x_{\pi(k)})).$$
\end{theorem}


Let $G=G_1 \Box G_2 \Box \dots \Box G_k$ be a prime factor decomposition of a connected graph $G$. Then, for each factor $G_i$ let the vertex set be $V(G_i) = \{x_{i1}, x_{i2}, \dots, x_{in_i}\}$, where $n_i$ is the order of the graph $G_i$. Then every vertex of the Cartesian product is of the form $(x_{1j_1},x_{2j_2}, \dots , x_{kj_k})$, where $x_{ij_i} \in V(G_i)$. Two vertices of the Cartesian product form an edge $$(x_{1j_1}, x_{2j_2}, \dots , x_{kj_k})(x_{1l_1}, x_{2l_2}, \dots , x_{kl_k})$$ if there exists exactly one index $i=1,\ldots,k$ such that $x_{ij_i} x_{il_i}$ is an edge of the factor $G_i$ and $x_{tj_t}=x_{tl_t}$ for all indices $t$ other than $i$. Given a vertex $v = (v_1 , v_2 , \ldots , v_k )$ of the product $G = G_1 \square G_2 \square \ldots \square G_k$, the \emph{$G_i$-layer through $v$} is the induced subgraph
$$G_{i}^{v}=G \left[ \{ x \in V (G) \ \vert \  p_{j}(x) = v_j \text{ for }j\neq i\}\right],$$

\noindent where $p_j$ is the projection mapping to the $j^\text{th}$-factor of $G$ \cite{hik}. It is clear that $G_{i}^{v}\simeq G_i$.

By \emph{$i^{\text{th}}$-quotient subgraph of $G$} we mean the graph
$$Q_{i}=G \diagup G_{i} \simeq G_1 \square \ldots \square G_{i-1} \square G_{i+1} \square \ldots \square G_k.$$
It is also evident that $G \simeq G_i \square Q_i$ \cite{hik}.

In Section 5 we make use of a tool developed by Boutin \cite{Butin-x, b1} called the determining set of a graph. 
A subset $S$ of the vertices of a graph $G$ is called a \emph{determining set} if whenever $g$ and $h$ are automorphisms of $G$ with the property that $g(s) = h(s)$ for all $s \in S$, then $g = h$. In particular, the following proposition gives a usefull characterization of a determining set.

\begin{proposition} \cite{Butin-x}
	Let $S$ be a subset of the vertices of the graph $G$. Then $S$ is a determining set for $G$ if and only if $\mathrm{Stab}(S) = \{ \id \}$.
\end{proposition}


 \section{The cost of edge-distinguishing of complete graphs}
 
 We recall that $D'(K_n)=2$ for $n\geq 6$. One way to prove this is to find an asymmetric spanning subgraph in $K_n$, see \cite{kp} and \cite{ac}. Forgetting about the trivial case $K_1$, the smallest asymmetric graph has six vertices while the smallest asymmetric tree has seven vertices, see Figures \ref{G61} and \ref{T71} \cite{Q}.
 
 \begin{figure}	[h!]
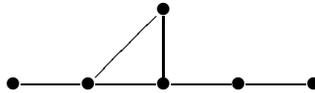

 	\[ \xygraph{
 		!{<0cm,0cm>; <1cm,0cm>:<0cm,1cm>::}
 		!{(1,1)}*{\bullet}="v1" !{(2,1)}*{\bullet}="v2"
 		!{(3,1)}*{\bullet}="v3" !{(3,2)}*{\bullet}="v4" 
 		!{(4,1)}*{\bullet}="v5" !{(5,1)}*{\bullet}="v6"
 		"v1"-@[black] "v2" "v2"-@[black] "v3" "v2"-@[black] "v4" 
 		"v3"-@[black] "v4" "v3"-@[black] "v5"
 		"v5"-@[black] "v6" 
 	} \]
 	\caption{The only asymmetric uni-cyclic graph on six vertices.}\label{G61}
 \end{figure}

\begin{figure}[h!]
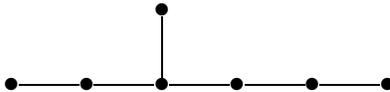
 
	\[ \xygraph{
		!{<0cm,0cm>; <1cm,0cm>:<0cm,1cm>::}
		!{(1,1)}*{\bullet}="v1" !{(2,1)}*{\bullet}="v2"
		!{(3,1)}*{\bullet}="v3" !{(3,2)}*{\bullet}="v4" 
		!{(4,1)}*{\bullet}="v5" !{(5,1)}*{\bullet}="v6"
		!{(6,1)}*{\bullet}="v7"
		"v1"-@[black] "v2" "v2"-@[black] "v3"  
		"v3"-@[black] "v4" "v3"-@[black] "v5"
		"v5"-@[black] "v6" "v6"-@[black] "v7"
	} \]
 \caption{The only asymmetric tree on seven vertices.}\label{T71}
\end{figure}
 
 If the red edges in a $2$-coloring of a complete graph on six or seven vertices induce these structures, then the coloring breaks all the automorphisms of $K_6$ and $K_7$. When $n \geq 8$, we can always have a vertex which is incident with only blue edges. Hence, we get $\rho' (K_6)=\rho' (K_7)=\rho' (K_8)=6$.
In general, the we obtain the following bounds on the cost of edge-distinguishing of complete graphs.
 
 \begin{theorem} \label{Kn}
 	Let $K_n$ be a complete graph on $n \geq 6$ vertices. Then $\rho'(K_n)$ is the minimum number of edges in an asymmetric graph on $n$ vertices. We have $n-2-\lfloor \frac{n-1}{7}\rfloor \leq \rho'(K_n) \leq n-2$ for $n\geq 8$. Moreover, $\rho'(K_n) \sim n$.
 \end{theorem}
 
 \begin{proof}
 	For $n\geq 9$, consider a coloring such that red edges form an asymmetric tree on $n-1$ vertices, while the remaining edges are colored blue. Such coloring breaks all non-trivial automorphisms of a complete graph. The upper bound on the cost of edge-distinguishing of a complete graph follows directly from this coloring. Obviously, $\rho'(K_n)$ is the number of edges in an asymmetric graph on $n$ vertices with minimum number of edges. And the lower bound comes from the fact that the smallest asymmetric tree has 7 vertices and the fact that if an asymmetric spanning subgraph on $n-1$ vertices has less than $n-2-\lfloor \frac{n-1}{7}\rfloor$ edges, then at least one of its connected components has less than 6 edges, which is impossible.
 	
 	Let us now assume that for infinitely many $n$ the inequality $\frac{\rho ' (K_n)}{n} \leq \frac{p}{q}$ holds for some integers $p$ and $q$ such that $p<q$. Furthermore, consider a distinguishing coloring of $K_n$ with $\rho'(K_n)$ red edges while the remaining edges are blue. Let $H$ be the graph induced by red edges on the vertices on $K_n$. Graph $H$ is almost spanning $K_n$ and is asymmetric. But we can assume that $H$ is spanning to make the calculations easier. Therefore, its components are also asymmetric and mutually non-isomorphic connected graphs. There are at least $n- \frac{p}{q}n = n(1 - \frac{p}{q})$ connected components in $H$. For if $H$ is a forest, then it has $n-\rho'(K_n) \geq n- \frac{p}{q} n$ components, and if $H$ contains a cycle, then it has more components than a forest with the same number of vertices and edges. Let $x$ be the average number of vertices in a connected component of $H$. Therefore,
 	$$x \leq \frac{n}{n (1-\frac{p}{q})} = \frac{q}{q-p}.$$
 	The number of asymmetric connected graphs with order bounded by a given constant is finite. However, the number of components of $H$ tends to infinity as $n$ grows to infinity. Therefore, for sufficiently large $n$, graph $H$ will have components of order large enough that the average number of vertices in a component of $H$ will exceed $\frac{q}{q-p}$. We arrive at a contradiction which concludes the proof.
 \end{proof}
 
 From Theorem \ref{Kn}, it can be deduced that for $8\leq n \leq 15$ we have $\rho'(K_n)=n-2$. However, we know that $\rho'(K_{16})=13$, as the following argument suggests. Cover vertices of $K_{16}$ by asymmetric trees of orders 1, 7 and 8, as Figure \ref{F16} suggests. Then these partition make a spanning asymmetric forest for $K_{16}$. This forest has 13 edges and so $\rho'(K_{16})\leq 13$. Since 12 edges in $K_{16}$ must have at lest 2 isomorphic components or a symmetric one, then it can be deduce that $\rho'(K_{16})>12$.
 
\begin{figure}[h!]
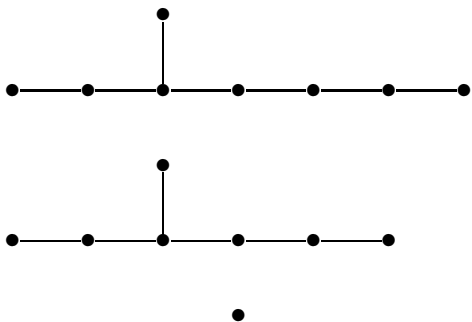
 
	\[ \xygraph{
		!{<0cm,0cm>; <1cm,0cm>:<0cm,1cm>::}
		!{(4,1)}*{\bullet}="v0"
		!{(1,2)}*{\bullet}="v1" !{(2,2)}*{\bullet}="v2"
		!{(3,2)}*{\bullet}="v3" !{(3,3)}*{\bullet}="v4" 
		!{(4,2)}*{\bullet}="v5" !{(5,2)}*{\bullet}="v6"
		!{(6,2)}*{\bullet}="v7"
		!{(1,4)}*{\bullet}="v8" !{(2,4)}*{\bullet}="v9"
		!{(3,4)}*{\bullet}="v10" !{(3,5)}*{\bullet}="v11" 
		!{(4,4)}*{\bullet}="v12" !{(5,4)}*{\bullet}="v13"
		!{(6,4)}*{\bullet}="v14" !{(7,4)}*{\bullet}="v15"
		"v1"-@[black] "v2" "v2"-@[black] "v3"  
		"v3"-@[black] "v4" "v3"-@[black] "v5"
		"v5"-@[black] "v6" "v6"-@[black] "v7"
		"v8"-@[black] "v9" "v9"-@[black] "v10"  
		"v10"-@[black] "v11" "v10"-@[black] "v12"
		"v12"-@[black] "v13" "v13"-@[black] "v14"
		"v14"-@[black] "v15"
	} \]
	\caption{An asymmetric forest on 16 vertices.}\label{F16}
\end{figure}
 
 We can think of a list of asymmetric trees, namely $\Omega = \{T_1 , T_2 , \ldots \}$, that for each $i=1, 2, \ldots$ gives an asymmetric tree so that for each $j>i$ we have $\vert T_j \vert \geq \vert T_i \vert$. The following procedure gives an asymmetric spanning forest in a complete graph.

 \begin{procedure}\label{pro}
 	For $n\geq 8$, to find an asymmetric red-edge-forest in the complete graph $K_n$, take the following steps.
 	\begin{itemize}
 		\item[Step 1.] Choose a vertex of $K_n$ and remove it from the set of avaiable vertices.
 		\item[Step 2.] Take the first asymmetric tree from $\Omega$ which is not already used, namely $T_i$. Choose $\vert T_i \vert$ vertices from available vertices of $K_n$ and color the edges of the subgraph induced by them red so that the subgraph induced by the red edges is isomorphic to $T_i$. Remove chosen vertices for the set of available vertices.   
 		\item[Step 3.] Check if the number of available vertices of $K_n$ is greater than the number of vertices of the first asymmetric tree from $\Omega$ which is not already used. If so, go to Step 2. Else, if the number of unchosen vertices of $K_n$ is zero proceed to Step 4, and if it is not zero, delete the last chosen red-edge asymmetric tree, put its vertices back to set of available vertices and choose a tree from $\Omega$, namely $T^\ast$, that has the same number of vertices that are available. Color some edges of the induced subgraph of the remaining set red so that the subgraph induced by these red edges is isomorphic to $T^\ast$ and proceed to Step 4.
 		
 		\item[Step 4.] Color all the remaining edges blue.
 	\end{itemize}
 	 
 \end{procedure}

Like Quintas \cite{Q}, by $a_n$ we mean the number of asymmetric trees having $n$ vertices, and for each integer $n\geq 8$ let $N$ and $w$ be defined by
$$\sum_{i=1}^{N} i a_n \leq n < \sum_{i=1}^{N+1} i a_{n},$$ and $$n=\sum_{i=1}^{N} i a_{n} + w(N+1)+r$$ where $0\leq w<a_{N+1}$ and $0\leq r<N+1$. Using these symbols, we can explicitly count the number of red-edges after the termination of Procedure \ref{pro}.
 
 \begin{theorem}
 	For any natural number $n\geq 8$, edge-cost of the complete graph $K_n$ is the number of red edges of an almost spanning asymmetric forest produced by Procedure \ref{pro}. In particular, for $n\geq 8$ we have $$\rho'(K_{n}) = n-\sum_{i=1}^{N} a_{n} -w.$$
 \end{theorem}
 
 \begin{proof}
 	Suppose that we acquire a coloring of edges of $K_n$ with red and blue so that number of red edges is less than the number of blue edges. If there are two or more vertices which are not incident to red edges, then their transposition is a color-preserving non-trivial automorphism.
 	
 	Observe that red edges induce some connected components in $K_n$. If two of these connected components are isomorphic, then the automorphism mapping one to another preserves the edge-coloring.
 	
 	Finally, every connected component induced by red edges has to be asymmetric because otherwise a non-trivial automorphism of a component can be easily extended to an automorphism of $K_n$. Since a tree has the least number of edges among any connected graphs with the same number of vertices, every connected component induced by red edges has to be an asymmetric tree.
 	
 	The ``in particular" part is evident from Theorem 1 of \cite{Q}.
 \end{proof}

In the Appendix, we have calculated the cost of edge-coloring for $K_n$ when $6\leq n\leq 630$. From this table we can see that sometimes $\rho'(K_n ) = \rho' (K_{n+1})$, for example when $n=6,7,15,24,33$ or so on.

\section{Reduced-factor Coloring of the Cartesian Product}

Let $G$ be a connected graph which has a prime factor decomposition $G=G_1 \square G_2 \square \ldots \square G_k$ for some $k\geq 2$, and let $f$ be a total coloring for $G$. Since a vertex or edge coloring can be easily transformed to a total coloring, everything here works well if $f$ is a vertex or edge coloring instead of a total coloring. This can be done by coloring all edges or all vertices of the graph $G$ with a fixed color.

For $i=1,\ldots,k$, we describe some total colorings for $G_i$, which we need in Lemma \ref{AUL} to find out whether $f$ is a distinguishing coloring. Let $V(G_{i})=\{ 1_{i},\ldots, m_i \}$. For each $j=1,...,m$, consider the layer graph isomorphic to $G_i$ consisting of vertices $$u_j = ( 1_{1},1_{2}, \ldots, 1_{i-1}, j_{i}, 1_{i+1},\ldots, 1_{k} ),$$
\noindent where $1_r$ is the first vertex of $G_r$ in our fixed ordering.

We define \emph{the total coloring of $Q_i^{u_j}$ (induced by $f$)}, denoted by $\check{Q}_{i}^{u_j}= (Q_{i}^{u_j},f)$, to be the graph $Q_{i}^{u_j}$ together with the total coloring induced by $f$. We say that the color $\check{Q}_{i}^{u_j}$ is \emph{equivalent} to $\check{Q}_{i}^{u_t}$ if there is a (total) color-preserving isomorphism $\varphi: Q_{i}^{u_j}\longrightarrow Q_{i}^{u_t}$.


Let $e=u_{i}v_i$ be an edge of $G_i$, the $i^{\mathrm{th}}$ factor of $G$. By $\overline{Q_i^e}$ we mean a vertex-colored graph isomorphic to $Q_i$ whose vertex set consists of edges of $G$ of the form $$(u_i, x)(v_i, x)=(x_{1j_1},\ldots, u_{i}, \dots , x_{kj_k})(x_{1j_1},\ldots, v_{i}, \dots , x_{kj_k})$$ for some $x\in V(Q_{i})$. Two vertices $(u_i, x)(v_i, x)$ and $(u_i, y)(v_i, y)$ are adjacent in $\overline{Q_i^e}$ if $x$ is adjacent to $y$ in $Q_i$. When each vertex $(u_i, x)(v_i, x)$ of $\overline{Q_i^e}$ is colored by $f((u_i, x)(v_i, x))$, the resulting vertex-coloring is called \emph{the vertex-coloring of $\overline{Q_i^e}$ (induced by $f$)}, and is denoted by $\hat{\overline{Q_i^e}}=(\overline{Q_i^e},f)$. The color $\hat{\overline{Q_i^e}}$ is \emph{equivalent} to $\hat{\overline{Q_i^{e'}}}$ if there is a vertex-color-preserving isomorphism $\vartheta:\overline{Q_i^e} \longrightarrow \overline{Q_i^{e'}}$.

Now, we can describe the colorings we need for lemma \ref{AUL}. For a (total) coloring $f$ of $G$, we describe a total coloring of $G_i$ by edge and/or vertex-colored $Q_{i}$s; color each vertex $u_j$ by $\check{Q}_{i}^{u_j}$ and color each edge $e=u_j u_{\ell}$ by $\hat{\overline{Q_i^{e}}}$. This total coloring of $G_i$ is called \emph{reduced factor coloring of $G_i$ induced by $f$} and is denoted $G_i^f$. We say $G_i^f$ is \emph{equivalent to} $G_j^f$ and write $G_i^f \simeq G_j^f$ if there is a total-color-preserving isomorphism from one to another.

\begin{lemma}\label{AUL}
Let $G=G_1 \square G_2 \square \ldots \square G_k$ with $k\geq 2$ be a connected graph decomposed into Cartesian prime factors and $f$ be a (total) coloring of $G$. If for each $i=1,\ldots, k$, the reduced-factor coloring $G_i^f$ is a distinguishing coloring and we have $G_i^f \not\simeq G_j^f$ for all $j\neq i$, then $f$ is distinguishing coloring of $G$.
\end{lemma}
\begin{proof}

Suppose that the condition holds. consequently, if $\varphi:G\longrightarrow G$ is a color-preserving automorphism, then, by Theorem \ref{autG}, there is a permutation $\pi$ of the set $\{1, 2, \dots , k\}$ and there are isomorphisms $\psi_i \colon G_{\pi(i)} \mapsto G_i$, such that
$$\varphi(x_1, x_2, \dots, x_k) = (\varphi_1 (x_{\pi(1)}), \varphi_2 (x_{\pi(2)}), \dots, \varphi_k (x_{\pi(k)}))$$ where $i=1, \dots, k$. Since we have $G_i^f \not\simeq G_j^f$ for all $j\neq i$, it can be deduced that $\pi$ is the identity permutation. Thus, $\varphi=(\varphi_1,\ldots,\varphi_k)\in\oplus_{i=1}^{k} \mathrm{Aut}(G_i)$. Now, for each $i=1,\ldots,k$, the automorphism $\varphi_i$ has to be $\id_{G_i}$ because $G_i^f$ is a total distinguishing coloring. As a result, $\varphi=\id_G$.
\end{proof}

\begin{remark}
As noted at the beginning of this section, Lemma \ref{AUL} must also be considered true whenever $f$ is a vertex or edge coloring. However, when $f$ is a vertex coloring of $G$, it is redundant to color each edge $e$ of $G_i$ with $\hat{\overline{Q_i^{e}}}$, because all edges of $G_i$ will receive the same color.
\end{remark}




\section{Some results}

In all of the following cases, we consider an edge or vertex coloring of a given Cartesian product with two colors, say red and blue. Assume that the number of blue elements in any considered coloring is greater than the number of red ones. Therefore, to prove any results for the cost or edge-cost, it suffices to minimize the number of red elements in the coloring and show that all nontrivial automorphisms are broken.

\begin{theorem}
Let $P_n$ be a path on $n$ vertices, where $n \geq 3$. If $k \leq 1 + \lfloor \frac{n}{2} \rfloor$, then $\rho ' (P_n^k) =1$. Moreover, if $t \leq \lfloor \frac{n}{2} \rfloor$, then $\rho (P_n^t) =1$.
\end{theorem}

\begin{proof} Since the graph $P_n^k$ is 2-edge-distinguishable, its edge-cost must be greater or equal to 1. Let us denote its consecutive factors by $P_n^{(1)}$, $P_n^{(2)}$, \dots, $P_n^{(k)}$. Color the vertex $$(x_{11}x_{21}x_{32} \dots x_{kk-1})$$ red while the other ones are blue.

All the reduced-factor colorings of $P_n^{(i)}$ has only one vertex which is colored differently from other vertices. For $P_n^{(i)}$-layer the color of vertex $x_{i(i-1)}$ is different from the other ones. Hence, for each $i=1,\ldots,k$, the reduced-factor coloring of $P_n^{(i)}$ are distinguishing. Meanwhile, for each pair of reduced-factor colored $P_n^{(i)}$ and $P_n^{(j)}$, they are not equivalent, therefore, there does not exist a permutation of the factors which would preserve the coloring of the Cartesian product. Then by Lemma \ref{AUL} all automorphisms of the Cartesian product are broken.

Now, to prove the implication for the edge-cost, color the edge $$(x_{11}x_{21}x_{32} \dots x_{kk-1})(x_{12}x_{21}x_{32} \dots x_{kk-1})$$ by red and let the remaining ones be blue. In the reduced-factor coloring of $P_n^{(1)}$ the edge $x_{11}x_{12}$is colored differently than the remaining ones. Therefore, this coloring is distinguishing. Together with a similar argument as above, we see that all the automorphisms of $G$ are broken.
\end{proof}

The case that all factors of the Cartesian product of paths are pairwise prime to each other is actually simpler.

\begin{theorem}
Let $G = \Box_{i=1}^k P_{n_i}$ be a Cartesian product of $k$ paths, where $n_i \neq n_j$ for $i \neq j$ and $n_i \geq 3$ for all $i \in \{1, \dots, k\}$. Then $\rho ' (G) =\rho (G)= 1$.
\end{theorem}

\begin{proof}
We define an edge coloring such that only the edge $(x_{11}x_{21} \dots x_{k1})(x_{12}x_{21} \dots x_{k1})$ is red, while the remaining edges of the Cartesian product are blue. We also define a vertex coloring such that $(x_{11}x_{21} \dots x_{k1})$ is the only red vertex while other vertices are all blue. These colorings are distinguishing because for every $i=1,\ldots, k$, the reduced-factor coloring of $P_{n_i}$ has only one vertex, namely $x_{i1}$, with different color from the other vertices. This guarantees that the reduced-factor colorings of $P_{n_i}$ are distinguishing. Since all paths are pairwise non-isomorphic, then there does not exist an automorphism that exchanges the factors. Therefore, the defined colorings are distinguishing by Lemma \ref{AUL}.
\end{proof}

The cost of edge-distinguishing for Cartesian powers of cycles can be treated in a similar way, as the following two theorems show. 

\begin{theorem} \label{Cnk}
Let $C_n$ be a cycle on $n$ vertices, where $n \geq 5$. If $2\leq k \leq 1 + \lfloor \frac{n}{2} \rfloor$, then $\rho ' (C_n^k) =2$. Moreover, when $t\leq \lfloor \frac{n}{2} \rfloor$ we have $\rho (C_n^t) =3$. 
\end{theorem}

\begin{proof}
Let us denote the consecutive factors $C_n^{(1)}$, $C_n^{(2)}$, \dots, $C_n^{(k)}$. We color edges $$(x_{11}x_{21}x_{31} \dots x_{k1})(x_{12}x_{21}x_{31} \dots x_{k1}),(x_{11}x_{22}x_{33} \dots x_{kk})(x_{11}x_{23}x_{33} \dots x_{kk})\in E(G)$$ by red, while the other edges are colored blue.

The reduced-factor coloring of $C_n^{(1)}$ has one edge $x_{11}x_{12}$ and one vertex $x_{11}$ in different color from the other ones. The reduced-factor coloring of $C_n^{(2)}$ has one edge $x_{22}x_{23}$ and one vertex $x_{21}$ in different color from the other ones. This two colorings are distinguishing for $C_n$ and they cannot be mapped onto each other. Furthermore, for the remaining reduced-factor colorings of $C_n^{(i)}$, the two vertices $x_{i1}$ and $x_{ii}$ receive different colors the other ones. The colorings $\check{Q_i}^{x_{i1}}$ and $\check{Q_i}^{x_{ii}}$ are equivalent. However, the isomorphism between them that preserves the colors permutes the first and the second layer. Since, this two layers are already distinguished from each other, such an isomorphism does not preserve the coloring of the Cartesian product. Therefore, we can assume that vertices $x_{i1}$ and $x_{ii}$ receive different colors. Since $k \leq 1 + \lfloor \frac{n}{2} \rfloor$, then the reduced-factor colorings $C_n^{(i)}$ are also distinguishing. The distance between the two vertices in different color from others is different for all $i \in \{3, 4, \dots, k\}$. Therefore, this colorings are pairwise non-isomorphic. It follows from the Lemma \ref{AUL} that all automorphisms of the graph $C_n^k$ are broken. Moreover, it can be easily checked that for any coloring such that only one edge is red there exists an automorphism that is not broken.

The implication about the cost for $t=1$ is clear while for $t\geq 2$ it follows from coloring vertices $$(x_{11}x_{21}x_{31} \dots x_{k1}), (x_{12}x_{21}x_{31} \dots x_{k1}),(x_{11}x_{22}x_{33} \dots x_{kk})\in V(G)$$ red while other vertices are blue. With a similar argument as above, we can easily see that when the number of factors do not exceed $\lfloor \frac{n}{2} \rfloor$ then it is a distinguishing coloring of the Cartesian product. Since $C_n^t$ is vertex-transitive, any coloring of $C_n^t$ with only two vertices colored red is not distinguishing. Hence, we are done with the proof.
\end{proof}

\begin{theorem}
Let $G = \Box_{i=1}^k C_{n_i}$ be the Cartesian product of $k$ cycles, where $k\geq 2 $ and $n_i \neq n_j$ for $i \neq j$ and $n_i \geq 5$ for all $i \in \{1, \dots, k\}$. Then $\rho ' (G) = 2$ and $\rho (G)=3$.
\end{theorem}

\begin{proof}
We define an edge-coloring that colors the edges $$(x_{11}x_{21}x_{31} \dots x_{k1})(x_{12}x_{21}x_{31} \dots x_{k1}),(x_{11}x_{22}x_{33} \dots x_{kk})(x_{11}x_{23}x_{33} \dots x_{kk})\in E(G)$$ red while the other ones are blue. Then the reduced-factor coloring of $C_{n_1}$ has only one edge $x_{11}x_{12}$ and only one vertex $x_{11}$ in different color to the other ones. The reduced-factor coloring of $C_{n_2}$ has also only one edge $x_{22}x_{23}$ and one vertex $x_{21}$ in different color to others. Therefore, in both cases the automorphisms of the cycle are broken. Moreover, all remaining reduced-factor colorings of $C_{n_i}$ have only two vertices $x_{i1}$ and $x_{ii}$ whose colors are different to all other vertices. This colors can be assumed to be different to each other because of the same reasoning as presented in the proof of Theorem \ref{Cnk}. This brings us to a conclusion that all of the possible automorphisms of the Cartesian product are broken. It is easy to see that a single red edge is not enough to break all of the automorphisms of the Cartesian product.

Moreover, we define a vertex coloring that makes $$(x_{11}x_{21}x_{31} \dots x_{k1}), (x_{12}x_{21}x_{31} \dots x_{k1}),(x_{11}x_{22}x_{33} \dots x_{kk})\in V(G)$$ red and leaves the others blue. With the same argument as above, this is a distinguishing coloring for $G$. Since $G$ is vertex-transitive, any $2$-coloring with only two red vertices cannot be a distinguishing one.
\end{proof}

One might ask what is the asymptotic behavior of the edge-cost of the Cartesian power when the number of factors tends to infinity? Can we find a graph $G$ such that $\lim \sup \rho' (G^ n )$ is bounded from above when $n$ tends to infinity? Using the following theorem we answer these questions. 

\begin{theorem}
Let $Q_k$ be the $k$-dimensional hypercube with $k \geq 3$. Then $\sup_k \rho ' (Q_k) = \infty$.
\end{theorem}

\begin{proof}
We denote the set of the vertices of the hypercube by the sequences of length $k$ with terms from the set $\{0,1\}$.
$$V(Q_k) = \{(a_1 a_2 \dots a_k) \colon a_i \in \{0,1\},\, i\in \{1, \dots, k\}\}$$
The edges are between two vertices differing in exactly one coordinate of the sequences.

Suppose that the cost of edge-coloring of the $k$-dimensional hypercube for all $k$ is bounded from above, i.e., $\sup_k \rho ' (Q_k) = n$ for some integer $n$. Then there exists an edge-coloring of $Q_k$ with exactly $n$ edges red (remaining edges are blue). Assume first that all of the red edges are in different $K_2$-layers. Let $k = n + 2^n + 1$. Consider all of the vertices of the $Q_k$ that are incident to red edges. At least two of their coordinates corresponding to the last $2^n + 1$ layers are exactly the same. Assume that they correspond to the $K_2^i$-layer and $K_2^j$-layer. Then an automorphism generated by a permutation $\pi =(ij) \in S_k$ preserves this coloring. Which is a contradiction.
\end{proof}

We may apply the same arguments to the Cartesian powers of any graph $G$. If $\sup_k \rho'(G^k) < \infty $, then for sufficiently large $k$ we find that two (or more) of the coordinates of all the vertices adjacent to red edges are exactly the same. This allows us to permute layers in which those vertices have the same coordinate.

\begin{corollary}
For any connected finite graph $G$ the edge-cost of the $n$-th Cartesian power of $G$ grows with the number of factors and tends to infinity, i. e.,
$$\sup_k \rho ' (G^k) = \infty.$$
\end{corollary}

Cartesian powers of $K_2$ need special attention. One reason is that we cannot break automorphisms of $K_2$ by edge coloring. 

\begin{lemma}
$\rho'(Q_3)=\rho'(Q_4)=3$.
\end{lemma}

\begin{proof}
In both cases we define a $2$-coloring with $3$ red edges that is distinguishing. In the case of $Q_3$ the red edges are $(0,0,0)(1,0,0)$, $(1,0,0)(1,1,0)$ and $(0,1,1)(1,1,1)$, as shown in Figure \ref{Q3}. It is easy to see that the reduced-factor colorings of all three factors are distinguishing and they are pairwise non-equivalent. Therefore, by Lemma \ref{AUL}, presented coloring of $Q_3$ is distinguishing. Moreover, since $Q_3$ is edge-transitive, for any coloring with two red edges there exists a non-trivial automorphism of $Q_3$ that maps one to another.

\begin{figure}[h!]
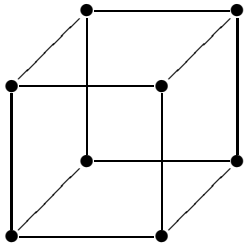
 \label{Q3}
	\[ \xygraph{
		!{<0cm,0cm>; <1cm,0cm>:<0cm,1cm>::}
		!{(1,3)}*{\bullet}="v13" !{(3,3)}*{\bullet}="v33"
		!{(0,2)}*{\bullet}="v02" !{(2,2)}*{\bullet}="v22" 
		!{(1,1)}*{\bullet}="v11" !{(3,1)}*{\bullet}="v31"
		!{(0,0)}*{\bullet}="v00" !{(2,0)}*{\bullet}="v20" 
		"v00"-@[red] "v20" "v00"-@[blue] "v11" "v00"-@[blue] "v02" 
		"v20"-@[red] "v22" "v20"-@[blue] "v31"
		"v11"-@[blue] "v31" "v11"-@[blue] "v13"
		"v02"-@[blue] "v13" "v02"-@[blue] "v22"
		"v33"-@[blue] "v31" "v33"-@[red] "v13" "v33"-@[blue] "v22"
	} \]
	\caption{A distinguishing coloring of $Q_3$}
\end{figure}

In case of $Q_4$ we color edges $(0,0,0,1)(1,0,0,0)$, $(0,0,0,0)(0,1,0,0)$ and $(0,0,0,1)(0,1,0,1)$ red. With a similar argument as previously it can be shown that this coloring is distinguishing and any coloring with only two red edges is not.  

\begin{figure}[h!]
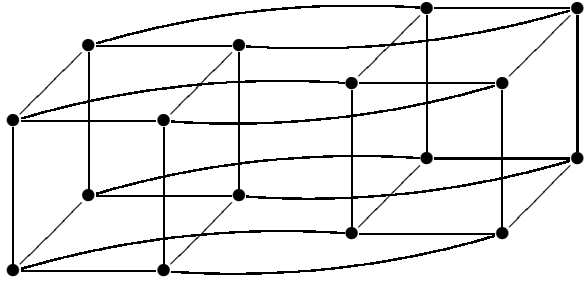
 \label{Q4}
	\[ \xygraph{
		!{<0cm,0cm>; <1cm,0cm>:<0cm,1cm>::}
		!{(1,3)}*{\bullet}="v13" !{(3,3)}*{\bullet}="v33"
		!{(0,2)}*{\bullet}="v02" !{(2,2)}*{\bullet}="v22"
		!{(1,1)}*{\bullet}="v11" !{(3,1)}*{\bullet}="v31"
		!{(0,0)}*{\bullet}="v00" !{(2,0)}*{\bullet}="v20" 
		!{(5.5,3.5)}*{\bullet}="v64" !{(7.5,3.5)}*{\bullet}="v84"
		!{(4.5,2.5)}*{\bullet}="v53" !{(6.5,2.5)}*{\bullet}="v73"
		!{(5.5,1.5)}*{\bullet}="v62" !{(7.5,1.5)}*{\bullet}="v82"
		!{(4.5,0.5)}*{\bullet}="v51" !{(6.5,0.5)}*{\bullet}="v71" 
		"v00"-@[red] "v20" "v00"-@[blue] "v11" "v00"-@[blue] "v02" 
		"v20"-@[red] "v22" "v20"-@[blue] "v31"
		"v11"-@[blue] "v31" "v11"-@[blue] "v13"
		"v02"-@[blue] "v13" "v02"-@[blue] "v22"
		"v33"-@[blue] "v31" "v33"-@[blue] "v13" "v33"-@[blue] "v22"
		"v51"-@[blue] "v71" "v51"-@[blue] "v62" "v51"-@[blue] "v53" 
		"v71"-@[red] "v73" "v71"-@[blue] "v82"
		"v62"-@[blue] "v82" "v62"-@[blue] "v64"
		"v53"-@[blue] "v64" "v53"-@[blue] "v73"
		"v84"-@[blue] "v82" "v84"-@[blue] "v64" "v84"-@[blue] "v73"
		"v00"-@[blue]@/^/ "v51" "v20"-@[blue]@/_/ "v71"
		"v11"-@[blue]@/^/ "v62" "v31"-@[blue]@/_/ "v82"
		"v02"-@[blue]@/^/ "v53" "v22"-@[blue]@/_/ "v73"
		"v13"-@[blue]@/^/ "v64" "v33"-@[blue]@/_/ "v84"
	} \]
	\caption{A distinguishing coloring of $Q_4$}
\end{figure}

\end{proof}

To proove a general upper bound on the cost of edge-distinguishing of hypercubes, we use the following theorem by Boutin \cite{b1}.

\begin{theorem} \cite{b1}
$Det(Q_n)= \lceil \log_2 n \rceil+1$.
\end{theorem}

In the proof of the above theorem, in the case $n=2^r$, Boutin explicitly defined a determining set of size $r+1$. It is easy to calculate that the distances between any two vertices of that set are equal to $2^{r-1}$. In the case of an arbitrary $n \geq 3$, the distances are not smaller than $2^{\lceil \log_2 n \rceil-2}$.

\begin{theorem}
Let $n \geq 5$, then $\rho '(Q_n) \leq \frac{\lceil \log_2 n \rceil + 4}{2} (\lceil \log_2 n \rceil + 1)$.
\end{theorem}

\begin{proof}
We consider the determining set $S =\{v_1, v_2, \dots, v_r\}$, where $r=\lceil \log_2 n \rceil  +1$, like in the proof of the theorem by Boutin. For each $i \in \{1,2, \dots , r\}$ we color $i+1$ edges incident with the vertex $v_i$ red . The remaining edges of the hypercube we color blue. We know that the vertices of the set $S$ are at distance greater than two from each other. Therefore, the red edges are adjecent only in vertices from the set $S$ and every vertex from the set $V(Q_n) \setminus S$ is incident with at most one red edge. 

Let us consider an automorphism $\varphi \in \Aut(Q_n)$ that preserves this coloring. Every vertex of the set $S$ has to be fixed by such an automorphism. Therefore, $\varphi \in Stab(S)$. Since $S$ is a detrmining set of the hypercube, $Stab(S) = \{\id\}$ and $\varphi = \id$. The number of edges colored red is equal to precisely $ \frac{\lceil \log_2 n \rceil + 4}{2} (\lceil \log_2 n \rceil + 1)$.
\end{proof}

We improve this bound further using the following result of Boutin \cite{b}.

\begin{theorem} \cite{b}
If $n \geq 5$, then $Q_n$ has a distinguishing class of size $2 \lceil \log_2 n \rceil -1$.
\end{theorem}

Again, Boutin \cite{b} presents a subset of the vertices of the hypercube that form this distinguishing class, together with the distances between them. We will make use of this particular set of vertices in the proof of the following theorem.lem 

\begin{theorem}
If $n \geq 5$, then $\rho '(Q_n) \leq 4 \lceil \log_2 n \rceil - 2$.
\end{theorem}

\begin{proof}
Let $S$ be the set of vertices that is the distinguishing class of size $2 \lceil \log_2 n \rceil -1$. For ever vertex in the set $S$ color two incident to it edges red. The remaining edges of the graph color blue. The distances of the vertices in the set $S$ are greater than two. Therefore, the red edges are adjacent only to the vertices of the set $S$. Consider $ \varphi$ to be an automorfism preserving this coloring. Clearly, it can only permute the vertices in $S$. Since $S$ is a disntinguishing class in a vertex coloring, $\varphi =\id$.
\end{proof}

\begin{proposition}
There exist Cartesian products of graphs for which $$\rho ' (G \Box H) < \max \{ \rho '(G), \rho '(H)\}.$$
\end{proposition}

The example below justifies this fact.

\begin{example}
Consider the Cartesian product of two non-isomorphic cycles on at least six vertices. Then $\rho ' (C_m \Box C_n) = 2$. However, $\rho ' (C_m) = \rho ' (C_n) =3$.
\end{example}

\begin{theorem}
Let $K_n^k$ be a Cartesian power of a complete graph. If $k\leq n+1$, then $ \rho' (K_n^k)\leq k\rho' (K_n)$.
\end{theorem}
\begin{proof}
	We define a distinguishing coloring of the graph $K_n^k$ with exactly $k\rho' (K_n)$ red edges. Let us fix an ordering of the vertices of $K_n$. In each factor we break all of its automorphisms by coloring $\rho' (K_n)$ edges red. For each $i=1,\ldots, k$ we choose the $i^{\textnormal{th}}$ vertex of the $i^{\textnormal{th}}$ factor $i=1,\ldots, n$ and color $\rho' (K_n)$ edges of the layer through $(x_{1i} x_{21} x_{31} \ldots x_{k1})$ red. By Lemma \ref{AUL}, we get a distinguishing coloring because all reduced-factor colorings are distinguishing and they are pairwise non-equivalent since colored edges are coming from different layers of the first factor.
\end{proof}

\pagebreak

\section{Appendix}
\begin{center}
\textbf{	Table of $\rho'(K_n)$}
\end{center}

\begin{table}[h]
	\centering\small
	\begin {tabular}{|c||c|||c||c|} \hline
	$n$&$\rho'(K_n)$&$n$&$\rho'(K_n)$\\ \hline\hline
	$6\leq n\leq 7$&6&$292\leq n\leq 303$ &$n-29$\\ \hline
	$8\leq n\leq 15$ &$n-2$&$304\leq n\leq 315$ &$n-30$ \\  \hline
	$16\leq n\leq 24$&$n-3$&$316\leq n\leq 327$ &$n-31$\\  \hline
	$25\leq n\leq 33$&$n-4$&$328\leq n\leq 339$ &$n-32$\\  \hline
	$34\leq n\leq 42$&$n-5$&$340\leq n\leq 351$ &$n-33$ \\ \hline
	$43\leq n\leq 52$&$n-6$&$352\leq n\leq 363$ &$n-34$ \\  \hline
	$53\leq n\leq 62$&$n-7$&$364\leq n\leq 377$ &$n-35$ \\ \hline 
	$63\leq n\leq 72$&$n-8$&$378\leq n\leq 389$ &$n-36$\\ \hline 
	$73\leq n\leq 82$&$n-9$&$390\leq n\leq 401$ &$n-37$\\\hline
	$83\leq n\leq 92$&$n-10$&$402\leq n\leq 413$ &$n-38$\\\hline
	$93\leq n\leq 102$&$n-11$&$414\leq n\leq 425$ &$n-39$\\\hline
	$103\leq n\leq 113$&$n-12$&$426\leq n\leq 437$ &$n-40$\\\hline
	$114\leq n\leq 124$&$n-13$&$438\leq n\leq 449$ &$n-41$\\\hline
	$125\leq n\leq 135$&$n-14$&$450\leq n\leq 461$ &$n-42$\\\hline
	$136\leq n\leq 146$&$n-15$&$462\leq n\leq 473$ &$n-43$\\\hline
	$147\leq n\leq 157$&$n-16$&$474\leq n\leq 485$ &$n-44$\\\hline
	$158\leq n\leq 168$&$n-17$&$486\leq n\leq 497$ &$n-45$\\\hline
	$169\leq n\leq 179$&$n-18$&$498\leq n\leq 509$ &$n-46$\\\hline
	$180\leq n\leq 190$&$n-19$&$510\leq n\leq 521$ &$n-47$\\\hline
	$191\leq n\leq 201$&$n-20$&$522\leq n\leq 533$ &$n-48$\\\hline
	$202\leq n\leq 212$&$n-21$&$534\leq n\leq 545$ &$n-49$\\\hline
	$213\leq n\leq 223$&$n-22$&$546\leq n\leq 557$ &$n-50$\\\hline
	$224\leq n\leq 234$&$n-23$&$558\leq n\leq 569$ &$n-51$\\\hline
	$235\leq n\leq 245$&$n-24$&$570\leq n\leq 581$ &$n-52$\\\hline
	$246\leq n\leq 256$&$n-25$&$582\leq n\leq 593$ &$n-53$\\\hline
	$257\leq n\leq 267$&$n-26$&$594\leq n\leq 605$ &$n-54$\\\hline
	$268\leq n\leq 279$&$n-27$&$606\leq n\leq 617$ &$n-55$\\\hline
	$280\leq n\leq 291$&$n-28$&$618\leq n\leq 630$ &$n-56$\\\hline
	\end {tabular}
	
	\medskip
	\caption{Some different values of $\rho'(K_n)$, calculated using Procedure \ref{pro} and a list of asymmetric trees on at most 12 vertices (e.g., see \cite{H}).	\label{tab_for_rho'_of_K_n}}
\end{table}


\begin{thebibliography}{99}


\bibitem{ac}
  M.~O.~Albertson and K.~L.~Collins,
  Symmetry breaking in graphs,
  Electron. J. Combin. 3 (1996) $\#$R18.
  
  \bibitem{Babai}
  L.~Babai,
  Asymmetric trees with two prescribed degrees,
  Acta Math. Acad. Sci. Hung. 29:1-2 (1977), pp. 193--200.

\bibitem{bc}
   B.~Bogstad and L.~Cowen,
   The distinguishing number of hypercubes,
   Discrete Math. 283 (2004), 29--35.
   
\bibitem{Butin-x} D.~Boutin,
   Identifying graph automorphisms using determining sets,
   Electron. J. Combin. 13 (2006) $\#$R78.

\bibitem{b}
   D.~Boutin,
   Small label classes in $2$-distinguishing labelings,
   Ars Math. Contemp. 1 (2008) 154--164.

\bibitem{b1}
   D.~Boutin,
   The determining number of a Cartesian product,
   Journal of Graph Theory (2008) 61 vol.2 77--87.

\bibitem{gkp}
  A.~Gorzkowska, R.~Kalinowski and M.~Pil\'sniak,
  The distinguishing index of the Cartesian product of finite graphs,
  Ars Math. Contemp. 12 (2017), 77--87.

\bibitem{eikpt}
	E.~Estaji, W.~Imrich, R.~Kalinowski, M.~Pil\'sniak, T.~Tucker,
	 Distinguishing Cartesian products of countable graphs,
	 Discuss. Math. Graph Theory 37 (2017), 155--164.

\bibitem{H}
   F. Harary and G. Prins,
   The number of homeomorphically irreducible trees and other species,
   Acta Math. 101, Number 1-2 (1959), 141--162.
	
\bibitem{hik}
   R. Hammack, W. Imrich and S. Klav\v{z}ar,
   Handbook of product graphs,
   Second edition, CRC Press, 2011.

\bibitem{ik}
  W.~Imrich and S.~Klav{\v{z}}ar,
  Distinguishing Cartesian powers of graphs,
  J. Graph Theory 53 (2006) 250--260.


\bibitem{kp}
  R.~Kalinowski and M.~Pil\'sniak,
  Distinguishing graphs by edge-colorings,
  European J. Combin. 45 (2015), 124--131.

\bibitem{kpw}
   R. Kalinowski, M. Pil\'{s}niak and M. Wo\'{z}niak,
   Distinguishing graphs by total colorings,
   Ars Math. Contemp. 11 (2016), 79--89.

\bibitem{p}
    M.\,Pil\'sniak,
    Improving upper bounds for the distinguishing index,
    Ars Math. Contemp. 13 (2017), 259--274.
    
\bibitem{Q}
    L. V. Quintas,
    Extrema concerning asymmetric graphs,
    J. Combinatorial Theory 3 (1967), 57--82.
    
\end{thebibliography}
\end{document}